\documentclass[11pt]{article}
\usepackage{amsfonts,amsmath,amssymb,amsthm, amscd}
\usepackage[dvips]{epsfig,graphics}

\numberwithin{equation}{section}

\newtheorem{theorem}{Theorem}[section]
\newtheorem{prop}[theorem]{Proposition}

\theoremstyle{definition}
\newtheorem{definition}[theorem]{Definition}
\newtheorem{example}[theorem]{Example}
\newtheorem{remark}[theorem]{Remark}

\def\<{{\langle}}
\def\>{{\rangle}}

\def\a{{\alpha}}

\def\e{{\epsilon}}
\def\g{{\gamma}}

\def\a{\alpha}

\def\L{{\Lambda}}

\def\e{\epsilon}
\def\A{{\cal A}}

\def\Z{\mathbb Z}
\def\C{\mathbb C}

\def\A{{\cal A}}

\begin{document}

\title{Alexander-Lin twisted polynomials}

\author{Daniel S. Silver \and Susan G. Williams\thanks{Both authors partially supported by NSF grant
DMS-0706798.} \\ {\em
{\small Department of Mathematics and Statistics, University of South Alabama}}}

\maketitle 

\begin{abstract} X.S. Lin's original definition of twisted Alexander knot polynomial is generalized for arbitrary finitely presented groups. J. Cha's fibering obstruction theorem is  generalized. The group of a nontrivial virtual knot shown by L. Kauffman to have trivial Jones polynomial is seen also to have a faithful representation that yields a trivial twisted Alexander polynomial. \end{abstract}

\section{Introduction.}  Twisted Alexander polynomials were introduced by X.S. Lin in a 1990 Columbia University preprint \cite{lin01}. Lin's invariant, defined in terms of Seifert surfaces, incorporated information from linear representations of the knot group $\pi$. While the classical Alexander polynomial is an invariant of the $\Z[t^{\pm 1}]$-module  $\pi'/\pi''$, its twisted counterpart can incorporate information from any term $\pi^{(k)}$ of the derived series of $\pi'$. As a result, it is often a more sensitive invariant. 

In \cite{wada94}, M. Wada modified Lin's original polynomial. He developed a general invariant for any finitely presented group provided with an epimorphism to a free abelian group and also a linear representation. The definition is given below. 

When applied to knot groups, Wada's invariant is often called the  ``twisted Alexander polynomial" (Twisted Reidemeister torsion is a more appropriate label. See \cite{kit96}.) A few years later, P. Kirk and C. Livingston, working from an algebraic topological perspective, offered another candidate for the title (see \cite{kl99}). 

We return to Lin's original invariant, putting it in Wada's general context. It is a ``pointed invariant," depending on a group and distinguished conjugacy class of elements. We generalize a fibering obstruction theorem of J. Cha. We also give an example of a virtual knot group and nonabelian faithful representation with trivial 
twisted invariants. 

For background on twisted Alexander invariants, see the comprehensive survey paper of S. Friedl and S. Vidussi \cite{fv09}.

\section{Alexander-Lin polynomials.}  

For any positive integer $d$, we regard $\Z^d$ as a multiplicative group $\<t_1, \ldots, t_d \mid t_it_j = t_j t_i\ (1 \le i, j \le d)\>$. 
The ring $\Z[\Z^d]$ of Laurent polynomials in variables $t_1, \ldots, t_d$ will be denoted by $\L$. When $d=1$, we write $t$ instead of $t_1$.

\begin{definition} \label{augmented group system} An {\it augmented group system}  is a triple $\A=(G, \e, x)$ consisting of a finitely presented group $G$, epimorphism $\e: G \to \Z^d$, and conjugacy class of an element $x \in G$ such that $\e(x)$ is nontrivial.  \end{definition}

Augmented group systems are objects of a category. A  morphism from $\A= (G, \e, x)$ to $\bar \A= (\bar G, \bar \e, \bar x)$ is defined to be a group homomorphism $f: G \to \bar G$ such that $f(x) = \bar x$ and $\e = \bar \e f.$ When $f$ can be made an isomorphism, we regard $\A$ and $\bar A$ as the same.  If $f$ can be made an epimorphism, we we write $\A \twoheadrightarrow \bar \A$. 

\begin{remark} If we restrict ourselves to groups that are hopfian (that is, every epimorphism from the group to itself is an automorphism), then the relation $\twoheadrightarrow$ is a partial order on augmented group systems (see \cite{sw06}). Knot and link groups are hopfian, as are any finitely generated, residually finite groups. 
 
Augmented group systems were introduced in \cite{silver93}. Definition \ref{augmented group system} is a mild generalization. 

\end{remark}

Let $\A=(G, \e, x)$ be an augmented group system, and let $\<x_0, x_1, \ldots, x_n \mid r_1, \ldots, r_m\>$ be a presentation of $G$. Without loss of generality we can assume that $x= x_0$ and also $m \ge n$.

{\sl Throughout, $R$ is assumed to be a Noetherian unique factorization domain. } 

Given a linear representation 
$\g: G \to {\rm GL}_NR$, mapping $x_i \mapsto X_i$, we define a representation $\g\otimes \e: \Z[G] \to M_N\L$ by sending $x_i \mapsto \e(x_i)X_i$, and extending. Here $M_N\L$ denotes  the ring of $N\times N$ matrices with entries in $\L$. Pre-compose with the natural projection $F=\<x_0, x_1, \ldots, x_n \mid \> \to G$ to obtain a homomorphism 
$$\Phi: \Z[F] \to M_N\L.$$ Let ${\cal M}_\g$ be the $m \times n$ matrix with $i, j$th entry equal to the image under $\Phi$ of the Fox partial derivative ${\partial r_i}/ {\partial x_j}$, denoted by
$$\bigg( {{\partial r_i}\over {\partial x_j}}\bigg)^\Phi \in M_N\L,$$
where $i=1, \ldots, n$ and $j=1, \ldots, m.$ Regard ${\cal M}_\g$ as an $mN \times nN$ matrix over $\L$ by ignoring inner parentheses.

\begin{definition} The {\it Alexander-Lin polynomial} $D_\g(\A)$ of 
$(G, \e, x)$ is  the greatest common divisor of all $nN \times nN$ minors  of ${\cal M}_\g$. \end{definition}

{\it A priori}, $D_\g(\A)$ depends on the presentation that we used. However, it is a direct consequence of \cite{wada94} that up to multiplication by a unit $\pm t_1^{n_1}\cdots t_d^{n_d}\in \L$, the polynomial $D_\g(\A)$ depends only on $\A$. (Furthermore, the sign $\pm$ is $+$ if $N$ is even.) Such invariance was proven in \cite{wada94} for the quotient $$W_\g=\frac{D_\g(\A)}{\det (I - x)^\Phi},$$
now known as {\it Wada's invariant}.  Since the denominator depends only on the conjugacy class of $x$, invariance holds for $D_\g(\A)$ as well.

In a similar way, the divisibility results of \cite{ksw05} proved for Wada's invariant give the following.

\begin{prop} Let $\A= (G, \e, x)$ and $\bar \A = (\bar G, \bar \e, \bar x)$ be augmented group systems such that $\A \twoheadrightarrow \A$ via $f: G \to \bar G$. Then for any representation $\g: \bar G \to {\rm GL}_NR$, 
$D_{\g f}(\A)$ is divisible by $D_\g(\bar \A)$. 
\end{prop}

\begin{example} If $k$ is an oriented knot with group $\pi_k$, then there 
is a well-defined epimorphism $\e: \pi_k \to \Z = \<t\mid \>$ taking each oriented meridian to $t$.  We associate to $k$ an augmented group system $\A_k = (\pi_k, \e, x)$, where $x$ represents the class of an oriented meridian. If $\g: \pi \to {\rm GL}_N \C$, then $D_\g(\A_k)$ agrees with the invariant originally defined by Lin \cite{lin01}. Lin used a Seifert surface $S$ that is {\it free} in the sense that ${\cal S}^3$ split along $S$ is a handlebody $H$ with fundamental group $\<a_1, \ldots, a_{2g}\mid \>$. He then described the knot group $\pi_k$ as 
$$\<x, a_1, \ldots, a_{2g}\mid x u_ix^{-1}  = v_i\  (i =1, \ldots, 2g)\>,$$
where $u_i, v_i$ are words in $a_1^{\pm 1}, \ldots, a_{2g}^{\pm 1}$ describing positive and negative push-offs of simple closed curves representing a basis for $\pi_1S$. Lin's invariant is the determinant of the matrix $M_\g$ computed as above.   

\end{example}

The following is a group-theoretic version of 3-manifold fibering results proved in different degrees of generality by Cha \cite{cha03}, Kitano and Morifuji
\cite{km05}, Goda, Kitano and Morifuji \cite{gkm05}, Friedl and Kim \cite{fk06}, Pajitnov \cite{paj07}, and Kitayama \cite{kit08}.

\begin{theorem}\label{fibering} Let $\A= (G, \e, x)$ be an augmented group system with $d=1$ and $\e(x)=t$. Assume such that the kernel of $\e:G \to \Z$ is a finitely generated free group of rank $n$. Then for any representation $\g: \pi \to {\rm G}_N R$, 
\item{1.} $\deg D_\g(\A) = nN$
\item{2.} the leading and trailing coefficients of $D_\g(\A)$ are units.
\end{theorem}

\begin{proof}  Let $a_1, \ldots, a_n$ be generators for the kernel of $\e$. Since $G$ is a semi-direct product of $\Z \cong \<x \mid \>$ and  the free group $F= \<a_1, \ldots, a_n \mid \>$, it has a presentation of the form 
$$G= \<x, a_1, \ldots, a_n \mid x a_1 x^{-1} = \mu(a_1), \ldots, \ x a_n x^{-1} = \mu(a_n)\>,$$
for some automorphism $\mu$ of  $F$.  The 
matrix ${\cal M}_\g$ in the definition of $D_\g(\A)$ is equal to
$${\cal M}_\g = t\  {\rm diag}(X, \ldots, X) - \bigg( {{\partial \mu(a_i)}\over {\partial a_j}}\bigg)^\Phi,$$
where $X = \g(x)$. It follows that the leading and constant terms of $D_\g(\A)=\det {\cal M}_\g$ are, respectively,
$$ t^{nN}(\det X)^n, \quad  \pm \det \bigg( {{\partial \mu(a_i)}\over {\partial a_j}}\bigg)^\Phi.$$
It suffices to show that the constant term is a unit.
By a theorem of Birman \cite{birman73}, elements $y_1, \ldots, y_n$ in $F$ generate if and only if the $n\times n$ Jacobian matrix with $ij$th entry ${\partial y_i}/ {\partial x_j}$ has a right inverse in $M_n(\Z[F])$. Hence
$$J=\bigg( {{\partial \mu(a_i)}\over {\partial a_j}}\bigg)$$
has a right inverse $B\in M_n(\Z[F])$. Replacing each occurrence of $a_i^{\pm 1}$ in $B$ by the matrix $A_i^{\pm 1}$, for $i=1, \ldots, n$, produces a right inverse for $J^\Phi$. Hence its determinant is a unit. 
\end{proof} 

\begin{example} The group $\pi_k$ of the trefoil has presentation 
$$\<x, a \mid x^2a x^{-2} \cdot x a^{-1} x^{-1} \cdot a \>.$$
A presentation for the commutator subgroup $\pi_k'$ can easily be found. Denoting $x^j a x^{-j}$ by $a_j$, 
$$\pi_k' =\< a_j \mid a_{j+2} a_{j+1}^{-1} a_j\ (j\in \Z) \>.$$
From this one sees that $\pi_k' \cong \<a_0, a_1 \mid\>$ is a finitely generated free group. 

Consider the related group $G$ with presentation

$$G= \<x, a \mid x^2a x^{-2}\cdot a^2 \cdot x a^{-1} x^{-1} \cdot a^{-1} \>.$$

The $\Z[t^{\pm 1}]$-modules $\pi'_k/\pi_k''$ and  $G'/G''$ are isomorphic. In particular, the Alexander polynomial is $t^2-t+1$ in both cases. However, using Magnus's Freiheitsatz as in \cite{str75}, one sees that $G'$ is not a finitely generated free group. Theorem \ref{fibering} provides a more elementary way of seeing this. 

Consider the augmented group system $\A = (G, \e, x)$, where $\e: G \to \Z$ is the abelianization map sending $x$ to $1$. One checks that there exists a  representation $\g: G \to {\rm GL}_3(\Z)$ such that 
$$\g(x)= \begin{pmatrix} 1&0&0\\ 0&0&1\\0&1&0 \end{pmatrix},\quad
\g(a)=  \begin{pmatrix} 0&0&1\\ 1&0&0\\0&1&0 \end{pmatrix}.$$ A simple calculation shows that 
$D_\g(\A) = (t^2-4)(t^2-1)(t^2-t+1)$. Since the constant term is not a unit, 
$G'$ is not a finitely generated free group. 

\end{example}

\section{Virtual knot example.}

Virtual knots were introduced by L. Kauffman  in 1997. They are described by diagrams that can have ``virtual crossings" (indicated by circles) as well as classical over/under crossings. A virtual knot is an equivalence class of diagrams under generalized Reidemeister moves. Details about virtual knots and their groups, a generalization of the classical notions, can be found in \cite{kauf01}. 

Many invariants of classical knots, such as the knot group and Jones polynomial, are defined in a natural way for virtual knots. Details can be found in \cite{kauf00}.

\begin{example} \label{virtual} Consider the virtual knot $k$ in Figure \ref{trefoil} with 
generators for a Wirtinger presentation indicated. In \cite{kauf00} Kauffman gave it as an example of a virtual knot with trivial Jones polynomial. We show that it also has a faithful representation with trivial Wada invariant. 

\begin{figure}
\begin{center}
\includegraphics[height=2 in]{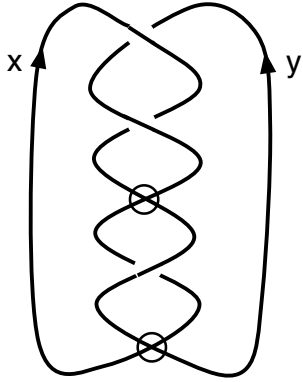}
\caption{Virtual knot and Wirtinger generators}
\label{trefoil}
\end{center}
\end{figure}

The group $\pi$ of $k$ has presentation
$$\pi= \< x, y \mid xyx=yxy, xy^{-1}x = yxy^{-1} \>.$$
Substituting $y = ax$ yields 
$$\pi = \< x, a \mid x^2 a x^{-2} = a^{-1} \cdot xax^{-1}, x a x^{-1} = a^2\>,$$
from which we can compute a presentation of the commutator subgroup: 
$$\pi' = \< a_j \mid a_{j+2} = a_j^{-1}a_{j+1}, a_{j+1} = a_j^2\ (j \in \Z)\>,$$ where $a_j = x^j a x^{-j}$. 
It follows immediately that $\pi' \cong \<a_0 \mid a_0^3\>$ is a cyclic group of order 3.
Hence $\pi$ is a semi-direct product of $\Z$ and $\Z/3\Z$: 
$$\pi = \<x, a \mid xa=a^2 x , a^3 =1\>.$$
Consider the representation $\g: \pi \to {\rm GL}_2\C$ sending $x$ and $a$ to 
$$X= \begin{pmatrix} 0&\a\\ \a&0 \end{pmatrix},\quad
A=  \begin{pmatrix} \omega&0\\ 0&\bar \omega \end{pmatrix},$$
where $\a$ is a nonzero complex number while $\omega$ is a primitive 3rd root of unity. If $\a$ is not a root of unity, then the matrix $X$ has infinite order, and the representation $\g$ is faithful. 
A straightforward calculation
shows that 
$${\cal M}_\g = \begin{pmatrix} tX-I -A\\ I+A+A^2 \end{pmatrix} = \begin{pmatrix} -\omega-1 & \a t\\  \a t & -\bar \omega-1\\ 0&0\\0&0\end{pmatrix}.$$
The Alexander-Lin polynomial is $D_\g = 1-\a^2 t^2$. Since this is the same as 
$\det (I- tX)$, the Wada invariant $W_\g$ is equal to 1. 
\end{example} 

\begin{remark}  In \cite{suz04}, M. Suzuki showed that the braid group $B_4$ with the faithful Lawrence-Krammer representation also has trivial  Wada invariant. The group in Example \ref{trefoil} is well known to be isomorphic to the group of a 2-knot, the 2-twist-spun trefoil. However, neither $B_4$ nor the group in Example \ref{trefoil} has a presentation with deficiency 1 (that is, one more generator than the number of relators).
\end{remark}

\noindent {\bf Question:} Does there exist an infinite nonabelian group $G$ of deficiency 1 and abelianization $\Z$ together with a faithful representation $\g: G \to {\rm GL}_n\C$ such that the Wada invariant $W_\g$ is equal to 1?  

\end{document}